\documentclass{article}
\usepackage{amsfonts}
\usepackage{amsmath}
\usepackage{amssymb}
\usepackage{amsthm}
\usepackage[english]{babel}
\usepackage{enumitem}
\usepackage{graphicx}
\usepackage{hyperref}

\newtheorem{theorem}{Theorem}[section]
\newtheorem{proposition}[theorem]{Proposition}

\newtheorem{corollary}[theorem]{Corollary}
\newtheorem{remark}[theorem]{Remark}
\theoremstyle{remark}

\newcommand{\minus}{\scalebox{0.60}[0.75]{$-$}} % custom minus sign

\title{Elementary Proof of a Theorem of Hawkes, Isaacs and \"Ozaydin}
\date{}
\author{Matth\'e van der Lee}

\begin{document}

\maketitle

\begin{abstract}
We present an elementary proof of the theorem of Hawkes, Isaacs and \"Ozaydin, which states that $\Sigma\,\mu_{G}(H,K)\equiv 0$ mod $d$, where $\mu_{G}$ denotes the M\"obius function for the subgroup lattice of a finite group $G$, $H$ ranges over the conjugates of a given subgroup $F$ of $G$ with $[G:F]$ divisible by $d$, and $K$ over the supergroups of the $H$ for which $[K:H]$ divides $d$. We apply the theorem to obtain a result on the number of solutions of $|\langle H,g\rangle|\mid n$, for said $H$ and a natural number $n$. The present version of the article includes an additional result on a quantity studied by K.S. Brown.
\end{abstract}

\begin{small}

\end{small}

\section{Introduction}

The purpose of this note is to present a simple proof for the theorem of Hawkes, Isaacs and \"Ozaydin, an important tool in the area of counting problems in finite groups.

We refer to \cite{HIO}, where the result has appeared as Theorem 5.1, for background information regarding the subject.

The exposition uses minimal group theory. The ingredients for the proof are Burnside's Lemma, the incidence algebra of a finite partially ordered set (of which we use only the basic properties), and a feature of arithmetic functions: Corollary \ref{cor:equiv2}.

Next, we derive a result on the number of solutions $g$ of $|\langle H,g\rangle|\mid n$, for a finite group $G$, a subgroup $H\leq G$, and $n\mid|G|$ (Theorem \ref{th:Frobenius}), and some results on a quantity studied by K.S. Brown (Propositions \ref{prop:phiinv} and \ref{prop:invphi}).

\section{A group action}\label{sec:action}

For positive integers $m$ and $d$, we define the number $\boldsymbol{b_{m}(d)}$ as the binomial coefficient
\[
b_{m}(d) = {dm\minus 1 \choose d\minus 1}.
\]
This function will play a pivotal role.

Let $t,m$ be positive integers, and $G$ a finite group of order $t\cdot m$. $G$ acts on the set $X = \{T\subseteq G\mid|T| = t\}$ of its subsets of cardinality $t$ in the natural way
\[
g\cdot T := gT = \{gx\mid x\in T\}.
\]
This is in fact the group action used in Wielandt's proof of Sylow's theorem, see \cite{HUP}, §1.7.

If $g\in G$ fixes an element $T\in X$, then $HT=T$ where $H=\langle g\rangle$ is the subgroup generated by $g$. So $T$ must be a union of right cosets of $H$ in $G$. As $|H|$ equals the order $\boldsymbol{o(g)}$ of $g$ in $G$, this means $o(g)$ divides $|T|=t$.
Writing $t=o(g)\cdot d$, we find that $T$ is the union of $d$ of the $d\cdot m$ right cosets of $H$. Conversely, any such union is an element of $X$ and is fixed by $g$. It follows that the number of fixed points in $X$ of an element $g\in G$ of order $t/d$ equals
\[
{dm\choose d}={dm\minus 1 \choose d\minus 1}\cdot m=b_{m}(d)\cdot m.
\]

And, clearly, if $o(g)$ does not divide $t$, $g$ cannot have any fixed points in $X$.

By Burnside's Lemma, the number of orbits of $G$ in $X$ is equal to the average number of fixed points of the elements of $G$. Denoting the number of elements of order $n$ in $G$ by $\boldsymbol{\chi(n)}=\chi_{G}(n)$, the number of orbits is therefore equal to
\[
(\Sigma_{d\mid t}\,b_{m}(d)\cdot m\cdot \chi(t/d))\,/\,(t\cdot m).
\]
Because this must be an integer, we have the following result.

\begin{proposition}\label{prop:chi}
If $m$ and $t$ are positive integers and $G$ is a group of order $tm$, one has
\begin{flalign*}
&&\Sigma_{d\mid t}\,b_{m}(d)\,\chi_{G}(t/d)\equiv \,0 \text{ mod } t &&\qed
\end{flalign*}
\end{proposition}

\section{Arithmetic functions modulo $n$}\label{sec:arith}

Let $n$ be a fixed positive integer, and consider the set $\boldsymbol{Ar_{n}} = \{f\mid f:\mathbb{N}\to\mathbb{Z}/n\}$ of functions from $\mathbb{N}$ to $\mathbb{Z}/n$. We can think of the standard arithmetic functions such as the Euler $\varphi$ and M\"obius $\mu$ functions as elements of this set, taking their values modulo $n$.

The set $Ar_{n}$ is a commutative ring under the operations of point-wise addition and convolution product multiplication
\begin{align*}
f+h:a&\mapsto f(a)+h(a)\\
f\,*\,h:a&\mapsto\Sigma_{d\mid a}\,f(d)\,h(a/d)
\end{align*}
Unity is the function $\boldsymbol{\varepsilon}$, given by $\varepsilon(1)=1$ and $\varepsilon(a)=0$ for $a>1$. The group of units of $Ar_{n}$ is:
\begin{align*}
Ar_{n}^{*} = \{f\in Ar_{n}\mid f(1)\in (\mathbb{Z}/n)^{*}\}.
\end{align*}
Indeed, if $f*h=\varepsilon$, then $f(1)\,h(1)\equiv 1$ mod $n$ and $\forall_{a\in \mathbb{N},\,a>1}\,\Sigma_{d\mid a}\,f(d)\,h(a/d)\equiv 0$ mod $n$. So $h(1)$ must be the inverse of $f(1)$ in $(\mathbb{Z}/n)^{*}$, and the value of $h(a)$ for $a>1$ can recursively be determined in $\mathbb{Z}/n$ from the second congruence, in which it has coefficient $f(1)$.

By the well-known formula $\Sigma_{d\mid a}\,\mu(d)=\delta_{1,a}=\varepsilon(a)$, the inverse $\mu^{\minus 1}$ of the ordinary M\"obius $\mu$ function in the ring $Ar_{n}$ is the function $\boldsymbol{E}$ given by $E(a)=1$ for all $a\in \mathbb{N}$. Thus,
\begin{equation}\label{eq:Emu}
\mu=E^{\minus 1}
\end{equation}
We now discuss the set $\boldsymbol{R_{n}}=\{f\in Ar_{n}\mid\forall_{t\mid n}\,f(t)\equiv 0$ mod $t\}$ of the arithmetic functions mod $n$ that are, to give them a name, \textit{special}.

\begin{proposition}\label{prop:subring}
$R_{n}$ is a unitary subring of $Ar_{n}$, and one has $R_{n}^{*}=R_{n}\cap Ar_{n}^{*}$. That is, if $f\in R_{n}$ is a unit in $Ar_{n}$, it is a unit in $R_{n}$.
\end{proposition}
\begin{proof}
$R_{n}$ is closed under addition, and $\varepsilon\in R_{n}$. If $f$ and $h$ are special mod $n$ and $t$ is a divisor of $n$, each term in the sum $(f*h)(t)=\Sigma_{d\mid t}\,f(d)\,h(t/d)$ is divisible by $d\cdot (t/d)=t$, and $f*h$ is again special.

Finally, if $f\in R_{n}\cap Ar_{n}^{*}$, let $t\mid n$ with $t>1$. Then $f(1)$ is relatively prime to $n$, hence to $t$. Furthermore, $f(1)\,f^{\minus 1}(t)\equiv\minus \Sigma_{1<d\mid t}\,f(d)\,f^{\minus 1}(t/d)$ mod $n$, so certainly modulo $t$. By induction, all terms on the right-hand side can be assumed to be multiples of $d\cdot (t/d)=t$. Hence $f(1)\,f^{\minus 1}(t)$ is a multiple of $t$, and therefore so is $f^{\minus 1}(t).$
\end{proof}
\begin{corollary}\label{cor:equiv1}
Let $f,h\in Ar_{n}$, and assume $f*E\in R_{n}$. Then
\begin{enumerate}[label=\emph{(\alph*)}]
\item $h*\mu\in R_{n}\Rightarrow h*f\in R_{n}$
\end{enumerate}
If, moreover, $f\in Ar_{n}^{*}$, the following equivalences hold
\begin{enumerate}[resume*]
\item $h*\mu\in R_{n}\Leftrightarrow h*f\in R_{n}$
\item $h*E\in R_{n}\Leftrightarrow h*f^{\minus 1}\in R_{n}$
\end{enumerate}
\end{corollary}
\begin{proof}
If $h*\mu\in R_{n}$, then since $f*E$ is also in $R_{n}$, so is their product $h*\mu*f*E$. By equation \eqref{eq:Emu}, that is just $h*f$, giving (a).

Now if $f$ is a unit, so is $f*E$. As this function is in $R_{n}$, by Proposition \ref{prop:subring} its inverse $f^{\minus 1}*\mu$ is also in $R_{n}$. So if $h*f\in R_{n}$, the product $h*f*f^{\minus 1}*\mu=h*\mu$ is in $R_{n}$ too, establishing (b).

For (c), note that $h*E\in R_{n}\Leftrightarrow h*(f^{\minus 1}*f)*E\in R_{n}\Leftrightarrow h*f^{\minus 1}\in R_{n}$, the latter equivalence because of $f*E\in R_{n}^{*}$.
\end{proof}

Note that the Corollary applies in particular to $f=\varphi$. Indeed, as $\Sigma_{d\mid a}\,\varphi(d)=a$, one has $\varphi*E=I\in R_{n}$, where $\boldsymbol{I}$ denotes the \textit{identity function}, $I(a)=a$ (taken modulo $n$). And $\varphi\in Ar_{n}^{*}$ since $\varphi(1)=1$.

Viewing the binomial coefficient function $b_{m}:d\mapsto {dm\minus 1 \choose d\minus 1}$ introduced in Section \ref{sec:action} as an element of $Ar_{n}$, we now have

\begin{proposition}\label{prop:equivalence}
For $m\in\mathbb{N}$ and $h\in Ar_{n}$, the following are equivalent
\begin{enumerate}[label=\emph{(\alph*)}]
\item $h*E\in R_{n}$
\item $h*b_{m}\in R_{n}$
\end{enumerate}
\end{proposition}

\begin{proof} Let $t$ be a divisor of $n$. We apply Proposition \ref{prop:chi} to the cyclic group $G=\mathbb{Z}/tm$. It has $\chi_{G}(d)=\varphi(d)$ for every divisor $d$ of $tm$, and hence we obtain $\Sigma_{d\mid t}\,b_{m}(d)\,\varphi(t/d)\equiv 0 \,(t)$. Writing $b$ for $b_{m}$, it follows that $b*\varphi\in R_{n}$. By (b) of Corollary \ref{cor:equiv1}, applied to $f=\varphi$ and $h=b$, we find $b*\mu\in R_{n}$.

As $b(1)=1$, the function $b$ is a unit in $Ar_{n}$. Taking inverses, we get $b^{\minus 1}*E\in R_{n}$. Part (c) of the Corollary, for $f:=b^{\minus 1}$, then gives the desired equivalence $h*E\in R_{n}\Leftrightarrow h*b\in R_{n}$ for $h\in Ar_{n}$.
\end{proof}

A slightly stronger statement is:

\begin{corollary}\label{cor:equiv2}
For $h\in Ar_{n}$, the following are equivalent
\begin{enumerate}[label=\emph{(\roman*)}]
\item $\forall_{t\mid n}\,\Sigma_{d\mid t}\,h(d)\equiv \,0\:(t)$
\item $\forall_{t\mid n}\,\exists_{m\in\mathbb{N}}\,\Sigma_{d\mid t}\,h(d)\,b_{m}(t/d)\equiv \,0\:(t)$
\end{enumerate}
\end{corollary}

\begin{proof}
(i) just paraphrases (a) of the Proposition, and (ii) is a trivial consequence of (b).
To see that (ii) implies (i), take a divisor $t$ of $n$, and suppose one already knows that $\forall_{k\mid t,\,k<t}\,\Sigma_{d\mid k}\,h(d)\equiv 0\,(k)$. By the Proposition, $h*b_{m}\in R_{k}$ for all $m\in\mathbb{N}$ and all $k\mid t$ with $k<t$.
Using (ii), pick an $m$ such that $\Sigma_{d\mid t}\,h(d)\,b_{m}(t/d)\equiv 0\,(t)$. Then clearly $h*b_{m}\in R_{t}$ for this particular $m$. Applying Proposition \ref{prop:equivalence} again, one finds that $h*E\in R_{t}$, and so (i) holds for $t$ as well.
\end{proof}

\section{The incidence algebra of a finite poset}\label{sec:incalg}

Let $\langle P,\leq\rangle$ be a finite poset (partially ordered set), and $A$ a commutative ring. The \textit{incidence algebra} of $P$ over $A$ is the set of functions in two variables on $P$ with values in $A$, which can only assume non-zero values when the arguments are comparable:
\begin{align*}
\boldsymbol{W_{P,A}} = \{f:P\times P\to A\mid\forall_{x,\,y\,\in\,P}\,(x\nleq y \Rightarrow f(x,y)=0)\}.
\end{align*}

The set $W_{P,A}$ is a, generally non-commutative unitary ring under point-wise addition and \textit{convolution product} multiplication
\begin{align*}
(f+h)(x,y)\, &=\,f(x,y)+h(x,y)\\
(f*h)(x,y)\, &=\,\Sigma_{x\leq z\leq y}\,f(x,z)\,h(z,y)
\end{align*}

Unity is Kronecker's delta function, $\boldsymbol{\delta}(x,y)=1$ if $x=y$ and $\delta(x,y)=0$ otherwise.

$W_{P,A}$ is an $A$-module, with $a\in A$ acting by $(a\cdot f)(x,y)=a\cdot f(x,y)$. And $A\to W_{P,A}, \,a\mapsto a\cdot\delta,$ is a ring homomorphism making it an $A$-algebra. Its group of units is given by:
\begin{align*}
W_{P,A}^{*}=\{f\in W_{P,A}\mid\forall_{x\in P}\,f(x,x)\in A^{*}\}.
\end{align*}

For if $f(x,x)\in A^{*}$ for all $x\in P$, we may compute the values $h(x,y)\in A$ of the right inverse $h$ of $f$ (the $h$ with $f*h=\delta$) for a given $y\in P$ inductively for the $x\leq y$ by means of
\begin{align*}
h(y,y)&=f(y,y)^{\minus1}\\
h(x,y)&=\minus f(x,x)^{\minus1}\cdot\Sigma_{x<z\leq y}\,f(x,z)\,h(z,y)
\end{align*}

(This works because $P$ is finite.) Similarly, the values of the left inverse $g$ of $f$ are found, given $x$, using
\begin{align*}
g(x,x)&=f(x,x)^{\minus1}\\
g(x,y)&=\minus f(y,y)^{\minus1}\cdot\Sigma_{x\leq z<y}\,g(x,z)\,f(z,y)
\end{align*}

By associativity, $g=g*\delta=g*f*h=\delta*h=h$ is the unique two-sided inverse of $f$.

Key elements of $W_{P,A}$ are the following functions $\boldsymbol{\zeta}$, $\boldsymbol{\eta}$ and $\boldsymbol{\mu}$:
\begin{align*}
&\zeta(x,y) = \begin{cases} \,1 &\mbox{if } x\leq y\\ 
\,0 & \mbox{otherwise} \end{cases} &&\mbox{the \textit{zeta function} of } P\\
&\eta(x,y) = \begin{cases} \,1 &\mbox{if } x<y\\ 
\,0 & \mbox{otherwise} \end{cases} &&\mbox{the \textit{chain function} of } P\\
&\mu = \zeta^{\minus1} &&\mbox{the \textit{M\"obius function} of } P
\end{align*}
The latter has the following properties, resulting from these definitions

\begin{equation}\label{eq:mu}
\left.\begin{aligned}
&\mu(x,x)=1 && \quad\qquad\mbox{for all } x\in P\:\\
&\mu(x,y)=\minus \Sigma_{x<z\leq y}\:\mu(z,y)=\minus \Sigma_{x\leq z<y}\:\mu(x,z) && \quad\qquad\mbox{for }x<y\:\\
&\Sigma_{x\leq z\leq y}\:\mu(x,z)=0=\Sigma_{x\leq z\leq y}\:\mu(z,y) && \quad\qquad\mbox{for }x\neq y\:
\end{aligned}
\right\}
\end{equation}

Many combinatorial aspects of the poset $\langle P,\leq\rangle$ are reflected in the ring $W_{P,A}$. For instance, the number of elements of the \textit{interval} $[x,y]=\{z\in P\mid x\leq z\leq y\}$, and the number of chains of length $k\geq 0$ from $x$ to $y$ are given by, respectively
\begin{align*}
\zeta^{2}(x,y) &= |[x,y]|\\
\eta^{k}(x,y) &=|\{(z_{0},\cdots,z_{k})\in P^{k+1}\mid x=z_{0}< z_{1}<\cdots<z_{k}=y\}|
\end{align*}
As $<$-chains of length $|P|$ cannot exist, we have $\eta^{|P|}=0$. And since $\zeta=\delta+\eta$ and $\delta=1$, it follows that
\[
\mu=\zeta^{\minus1}=(\delta+\eta)^{\minus1}=\delta-\eta+\eta^{2}-\cdots\pm\eta^{|P|\minus1}
\]

\begin{remark}\label{rem:Hall}
By this formula, $\mu(x,y)$ is equal to (the image under $\mathbb{Z}\to A$ of) the number of even-length chains from $x$ to $y$ in $P$ minus the number of odd-length ones. (Here, $x=y$ is considered to form a $<$-chain of length zero.) This result is due to P. Hall.
\end{remark}

\section{Proof of the Theorem}\label{sec:proof}

Let $G$ be a finite group of order $|G|=t\cdot m$, and let $\boldsymbol{\mathcal{L}(G)}=\{H\mid H\leq G\}$ be its subgroup lattice. It is partially ordered by the inclusion relation $\subseteq$, which we shall denote as $\leq$. Using $A=\mathbb{Z}$ as base ring, we will write $\boldsymbol{\mathcal{W}(G)}$ for the associated incidence algebra $W_{\mathcal{L}(G),\mathbb{Z}}$.

We denote $\boldsymbol{\mu_{G}}$, the M\"obius function of the poset $\mathcal{L}(G)$, simply by $\mu$, and similarly write just $\zeta$ for the zeta function.

We revisit the action of $G$ on $X = \{T\subseteq G\mid |T| = t\}$ discussed in Section \ref{sec:action}. For $K\leq G$, define:
\begin{align*}
\Lambda(K) &= \{T\in X\mid G_{T}=K\}\\
\lambda(K) &= |\Lambda(K)|
\end{align*}

where $\boldsymbol{G_{T}}=\{g\in G\mid g\cdot T=T\}$ is the stabilizer of $T\in X$. If $H=G_{T}$, then $T=HT$ is a union of right cosets of $H$ in $G$, so $|H|$ divides $t$. Putting $d=t/|H|$, the size of the orbit $G\cdot T$ equals $[G:H]=dm$.

A point $gT$ in this orbit has stabilizer $G_{gT}=gHg^{\minus 1}$, a conjugate of $H$. So if $\mathcal{H}$ denotes the conjugacy class of $H$ in $G$, it follows that $\bigcup_{L\in\mathcal{H}}\Lambda(L)$ is a union of orbits of $G$ in $X$, each of length $dm$. As the $\Lambda(L)$ are disjoint, we obtain $\Sigma_{L\in\mathcal{H}}\,\lambda(L)\equiv \,0\:(dm)$.

Hence if $d\mid t$ and $\mathcal{E}$ is any set of subgroups of order $t/d$ of $G$ closed under conjugation, we have:
\begin{equation}\label{eq:alpha}
\Sigma_{H\in\mathcal{E}}\:\lambda(H)\equiv \,0\:(dm)
\end{equation}

Now let $H\leq G$ fix $T\in X$. Then $H\leq G_{T}$, so $|H|$ again divides $t$. With $d:=t/|H|$, $T$ must be a union of $d$ of the $d\cdot m$ right cosets of $H$. As we saw earlier, the number of these unions equals $b_{m}(d)\cdot m$, which is therefore the number of fixed points of $H$ in $X$. It follows that for any $H\leq G$ with $|H|\mid t$
\begin{equation}\label{eq:beta}
b_{m}(t/|H|)\cdot m=\Sigma_{H\leq K\leq G,\, |K|\mid t}\,\lambda(K)
\end{equation}

Indeed, $T$ is a fixed point of $H$ iff $H\leq G_{T}$, and $|G_{T}|$ divides $t$ for every $T\in X$.

Let $F$ be a subgroup of $G$, of order dividing $t$, and consider the collection
\[
\mathcal{F}=\{H\in \mathcal{L}(G)\mid H\,\text{contains a conjugate of}\,F\}.
\]
The remaining computations will take place in the incidence algebra $\mathcal{W}_{\mathcal{F}}:=W_{\mathcal{F},\mathbb{Z}}$ of the poset $\langle\mathcal{F},\leq\rangle$. As $\mathcal{F}$ is \textit{convex} in $\mathcal{L}(G)$, that is, for $H\leq L\leq K$ in $\mathcal{L}(G)$ with $H,K\in\mathcal{F}$ one has $L\in\mathcal{F}$, it is clear that, for $H$ and $K$ in $\mathcal{F}$, $\mu(H,K)$ has the same value whether taken in $\mathcal{W}(G)$ or in $\mathcal{W}_{\mathcal{F}}$. (See also Remark \ref{rem:Hall}.) The same goes for the zeta function, and so we can use the notations $\mu$ and $\zeta$ unambiguously.

In order to apply M\"obius inversion to (\ref{eq:beta}), we introduce two auxiliary functions $\alpha$ and $\beta$ from the ring $\mathcal{W}_{\mathcal{F}}$
\begin{align*}
\alpha(K,L) &= \begin{cases} \,\lambda(K) & \text{if } |K| \text{ divides } t \text{ and } L=G \\
\,0 & \text{otherwise}\end{cases}\\
\beta(K,L) &= \begin{cases} \,b_{m}(t/|K|) & \text{if } |K| \text{ divides } t \text{ and } L=G \\
\,0 & \text{otherwise} \end{cases}
\end{align*}

In terms of these, we have for $H,L\in\mathcal{F}$:
\[
m\cdot\beta(H,L)=\Sigma_{H\leq K\leq L}\,\alpha(K,L).
\]

For in case $L=G$ and $|H|\mid t$ this follows by (\ref{eq:beta}), and the formula is trivial otherwise.

Reformulating it as $m\cdot\beta(H,L)=\Sigma_{H\leq K\leq L}\,\zeta(H,K)\,\alpha(K,L)$, the equation shows that $m\cdot\beta=\zeta*\alpha$ in $\mathcal{W}_{\mathcal{F}}$. Inverting $\zeta$ now gives $\alpha=m\cdot(\mu*\beta)$. Hence, for any $H\in\mathcal{F}$ with $|H|\mid t$:
\[
\lambda(H)=\alpha(H,G)=m\cdot\Sigma_{H\leq K\leq G}\:\mu(H,K)\:\beta(K,G).
\]

By the definition of $\beta$ and formula \eqref{eq:alpha} applied to $\mathcal{E}=\{H\in\mathcal{F}\mid |H|=t/d\}$, one obtains, dividing out $m$
\begin{align*}
\forall_{d\mid t}\,\Sigma_{H\in\mathcal{F},\,|H|=t/d}\,\Sigma_{H\leq K\leq G,\,|K|\mid t}\:\mu(H,K)\:b_{m}(t/|K|)\equiv \,0\:(d).
\end{align*}

Collecting the $K$ that are of equal index over the corresponding $H$ and eliminating $m=|G|/t$ produces
\begin{align*}
\forall_{d\mid t}\,\Sigma_{k\mid d}\,(\Sigma_{H\in\mathcal{F},\,|H|=t/d}\,\Sigma_{H\leq K\leq G,\,[K:H]=k}\:\mu(H,K))\cdot b_{|G|/t}(d/k)\equiv \,0\:(d).
\end{align*}

This last formula is valid for any $t\mid |G|$ for which $f:=|F|\mid t$. So, given any divisor $d$ of $[G:F]$, one can employ the formula with $t=f\cdot d$. In this situation, the $H\in\mathcal{F}$ for which $|H|=t/d$ are just the elements of the conjugacy class $\mathcal{C}$ of $F$, and we arrive at:
\[
\forall_{d\mid [G:F]}\,\Sigma_{k\mid d}\,(\Sigma_{H\in\mathcal{C}}\,\Sigma_{H\leq K\leq G,\,[K:H]=k}\:\mu(H,K))\cdot b_{|G|/(fd)}(d/k)\equiv \,0\:(d).
\]
We are now in a position to apply Corollary \ref{cor:equiv2}, yielding
\[
\forall_{d\mid [G:F]}\,\Sigma_{k\mid d}\,\Sigma_{H\in\mathcal{C}}\,\Sigma_{H\leq K\leq G,\,[K:H]=k}\:\mu(H,K)\equiv \,0\:(d).
\]
And this is precisely the content of the theorem of Hawkes, Isaacs and \"Ozaydin.

\begin{theorem}\label{th:mutheorem}\textnormal{(\cite[Theorem 5.1]{HIO})}
If $G$ is a finite group, $F$ a subgroup, $\mathcal{C}$ the conjugacy class of $F$ in $G$, and $d$ a divisor of $[G:F]$, the following congruence applies
\begin{flalign*}
&&\Sigma_{H\in\mathcal{C}}\,\Sigma_{H\leq K\leq G,\,[K:H]\mid d}\:\mu_{G}(H,K)\equiv \,0 \text{ mod } d &&\qed
\end{flalign*}
\end{theorem}

Taking the sum over all conjugacy classes of subgroups of order $h\mid |G|$, and letting $d\mid |G|/h$, one has, as a consequence
\begin{equation}\label{eq:muthm}
\Sigma_{H\leq G,\,|H|=h}\,\Sigma_{H\leq K\leq G,\,[K:H]\mid d}\:\mu(H,K)\equiv \,0 \text{ mod } d
\end{equation}
Loosely speaking, the sum of the $\mu(H,K)$, taken over all $H$ "at level $h$" and all $K$ "at level at most $h\cdot d$", is a multiple of the "height $d$ of the resulting slice" of $G$.

\section{Further observations}\label{sec:further}

Write $\boldsymbol{\psi(n)}=\psi_{G}(n)=|\{H\leq G\mid |H|=n\}|$ for the number of subgroups of order $n$ in $G$.

Taking $d=p$, a prime, $h\in\mathbb{N}$ with $hp\mid |G|$, $H\in \mathcal{L}(G)$ with $|H|=h$, and $H\leq K\leq G$ with $[K:H]\mid p$, we either have $K=H$ and $\mu(H,K)=1$ by properties \eqref{eq:mu}, or $|K|=hp$ and $K$ \textit{covers} $H$ (in the sense that no intermediate subgroups $H<L<K$ exist), so that $\mu(H,K)=\minus1$, again by \eqref{eq:mu}. Hence, according to \eqref{eq:muthm}:
\[
\psi(h)-|\{(H,K)\mid H\leq K\leq G,\: |H|=h \:\text{and}\:[K:H]=p\}|\,\equiv \,0\:(p),
\]
or, putting $t=hp$,
\begin{equation}\label{eq:Sylow}
\forall_{t\mid|G|}\,\forall_{p\mid t,\,p\,prime}\: \psi_{G}(t/p)\,\equiv\,\Sigma_{K\leq G,\,|K|=t}\:\psi_{K}(t/p)\;(p)
\end{equation}

Sylow's theorem follows by an easy induction: if $1<t\mid |G|$ is a power of a prime $p$, and one assumes that $\psi_{K}(t/p)\equiv 1\,(p)$ for all $K\leq G$, the left-hand side of \eqref{eq:Sylow} and all terms on the right-hand side are congruent to 1 mod $p$, so that $\psi_{G}(t)$, being the number of terms on the right, must be $\equiv1$ mod $p$ as well.$\hfill\square$
\\

Note that Frobenius' theorem, by which $\Sigma_{d\mid t}\,\chi_{G}(d)\equiv 0\:(t)$ for a finite group $G$ of order divisible by $t$, follows from Proposition \ref{prop:chi}, with the aid of a little arithmetic. Indeed, $\forall_{t\mid|G|}\,\Sigma_{d\mid t}\,\chi(d)\,b_{|G|/t}(t/d)\equiv 0 \,(t)$ by Proposition \ref{prop:chi}, and Corollary \ref{cor:equiv2} applies.$\hfill\square$
\\

Next, we note that the proof of Proposition \ref{prop:equivalence} goes through for any function $b\in Ar_{n}$ (notation from that proof) which is a unit in $Ar_{n}$ and satisfies $b*\varphi\in R_{n}$, i.e., for any $b$ in the coset $\varphi^{\minus 1}*R_{n}^{*}$ of $R_{n}^{*}$ in $Ar_{n}^{*}$.
This coset equals $\mu^{\minus 1}*R_{n}^{*}=E*R_{n}^{*}$, by (b) of Corollary \ref{cor:equiv1} applied to $f=\varphi$. Thus, for any function $b$ in the coset and any function $h\in Ar_{n}$, there is an equivalence similar to the one in Proposition \ref{prop:equivalence}:
\[
\forall_{t\mid n}\,\Sigma_{d\mid t}\,h(d)\equiv \,0\:(t)
\quad\Leftrightarrow\quad
\forall_{t\mid n}\,\Sigma_{d\mid t}\,h(d)\,b(t/d)\equiv \,0\:(t).
\]
Such $b$ include the inverse of $\varphi$, given by $\varphi^{\minus 1}(a) = \prod_{\,p\mid a,\,p\,\text{prime}}\,(1\minus p)$ for $a\in\mathbb{N}$, and the function $\sigma(a)=\Sigma_{d\mid a}\,d$. For $\sigma=E*I$ and $I\in R_{n}^{*}$. Generally, putting $I_{k}(a):=a^{k}$ and $\sigma_{k}:=E*I_{k}$, so that $\sigma_{k}(a)$ is the sum of the $k$-th powers of the divisors of $a$, the equivalence holds for $b=\sigma_{k}$. Also, for $u\in\mathbb{Z}/n$ it is easily seen that $\Sigma_{d\mid n}\,u^d\cdot\varphi(n/d)\equiv \,0\:(n)$, so that the function $b:a\mapsto u^a$ is in the coset when $u\in(\mathbb{Z}/n)^{*}$.\\

We conclude by taking a look at the following elements $\boldsymbol{i}, \boldsymbol{\varphi}$ and $\boldsymbol{\gamma}$ of the incidence algebra $\mathcal{W}(G)$:
\begin{align*}
&i(H,K) = \begin{cases} \,[K:H] &\mbox{for } H\leq K\\ 
\,0 & \mbox{otherwise} \end{cases} &&\mbox{the \textit{index function} for } G\\
&\varphi = i*\mu &&\textit{Euler totient function} \mbox{ for } G\\
&\gamma(H,K) = |\{g\in K\mid K=\langle H,g\rangle\}|&&\mbox{number of }\textit{single generators} \mbox{ of } K \mbox{ over } H
\end{align*}

Then $i=\varphi\,*\,\zeta$, so for subgroups $H\leq L\leq G$ we have $[L:H]=\Sigma_{H\leq K\leq L}\,\varphi(H,K)$, and hence
\begin{equation}\label{eq:phi}
|L|=|H|\cdot\Sigma_{H\leq K\leq L}\,\varphi(H,K)
\end{equation}

It follows that for any $H,K\in\mathcal{L}(G)$:
\begin{equation}\label{eq:gamma}
\gamma(H,K)=|H|\cdot\varphi(H,K)
\end{equation}

For both sides vanish when $H\nleq K$. We use induction on $[K:H]$ for the case $H\leq K$. If $[K:H]=1$, both sides are equal to $|H|$. And assuming the equation holds for all $K$ with $H\leq K<L$, by (\ref{eq:phi}) we find:
\[
|L|\,=\,|H|\cdot\varphi(H,L)+\Sigma_{H\leq K<L}\,\gamma(H,K)\,=\,|H|\cdot\varphi(H,L)+|\{g\in L\mid \langle H,g\rangle<L\}|.
\]
But $|L|\,=\,|\{g\in L\mid \langle H,g\rangle\leq L\}|$, and therefore $|H|\cdot\varphi(H,L)\,=\,|\{g\in L\mid \langle H,g\rangle=L\}|\,=\,\gamma(H,L).$\\

The theorem of Frobenius mentioned above is the case $F=1$ of:

\begin{theorem}\label{th:Frobenius} If $G$ is a finite group, $F\leq G$, $\mathcal{C}$ the conjugacy class of $F$ in $G$, and $n$ divides $|G|$, one has
\[
\Sigma_{H\in\mathcal{C}}\,\Sigma_{H\leq K\leq G,\,|K|\mid n}\:\gamma(H,K)\equiv\,0\text{ mod } n.
\]
That is, $n$ divides $|\mathcal{{C}}|$ times the number of $g\in G$ for which $|\langle F,g\rangle|\mid n$.
\end{theorem}

\begin{proof}
Using (\ref{eq:gamma}) and the definition of $\varphi$,
\begin{align*}
&\Sigma_{H\in\mathcal{C}}\,\Sigma_{K\geq H,\,|K|\mid n}\,\gamma(H,K)=\\
&|F|\cdot\Sigma_{H\in\mathcal{C}}\,\Sigma_{K\geq H,\,|K|\mid n}\,\varphi(H,K)=\\
&|F|\cdot\Sigma_{H\in\mathcal{C}}\,\Sigma_{K\geq H,\,|K|\mid n}\,\Sigma_{H\leq L\leq K}\,[L:H]\cdot\mu(L,K)=\\
&\Sigma_{H\in\mathcal{C}}\,\Sigma_{L\geq H,\,|L|\mid n}\,|L|\cdot(\Sigma_{K\geq L,\,|K|\mid n}\,\mu(L,K)).
\end{align*}

For a fixed value of $|L|$, the $L$ appearing in the final expression form a family closed under conjugation, so by Theorem \ref{th:mutheorem} the corresponding terms add up to a multiple of $|L|\cdot(n/|L|)=n$.
\end{proof}
The number of $n$-tupels $(g_{1},\cdots,g_{n})\in G^{n}$ for which $G=\langle g_{1},\cdots,g_{n}\rangle$ is traditionally denoted $\varphi_{G}(n)$ (\cite{GAS}). In our notation it is $\gamma^{n}(1,G)$. We use the symbol $\varphi$ for a related but different notion. Writing $\boldsymbol{\varphi(K)}$ for $\varphi(1,K)$, equation (\ref{eq:phi}), applied to $H=1$ and $L=G$, gives $|G|=\Sigma_{K\leq G}\,\varphi(K)$, generalizing the household formula $\Sigma_{d\mid n}\,\varphi(d)=n$ (which is the special case $G=\mathbb{Z}/n$). When $K$ is cyclic, one has $\varphi(K)=\varphi(|K|)$, the ordinary Euler $\varphi$-function applied to the group order $|K|$, and $\varphi(K)=0$ otherwise.
\\

The smallest number of generators of a group $L$ over a subgroup $H$ is the least $n$ such that $\varphi^n(H,L)>0$, as follows by induction from \eqref{eq:gamma}, using the fact that, again by \eqref{eq:gamma}, $\varphi^n(H,L)$ is always non-negative.\\

The inverse of the index function $i$ in $\mathcal{W}(G)$ is given by
\begin{equation}\label{eq:coindex}
i^{\minus 1}(H,L)=i(H,L)\,\mu(H,L)
\end{equation}
For we have $\Sigma_{H\leq K\leq L}\,i(H,K)\cdot i(K,L)\,\mu(K,L)=[L:H]\cdot\Sigma_{H\leq K\leq L}\,\mu(K,L)=[L:H]\cdot\delta(H,L)$ by \eqref{eq:mu}, and that simply equals $\delta(H,L)$.\\

From \eqref{eq:coindex} and the definition of $\varphi$ one obtains:
\begin{equation}\label{eq:phiinv}
\varphi^{\minus 1}(H,L)=(\zeta*i^{\minus 1})(H,L)=\Sigma_{H\leq K\leq L}\,[L:K]\,\mu(K,L)
\end{equation}
It is an open problem, raised by Kenneth Brown (\cite{BCP}), whether $\varphi^{\minus 1}(1,G)$ can ever be zero for a finite group $G$.
Gasch\"utz (\cite{GAS}) has shown that for solvable $G$ it cannot, and Brown has derived some interesting divisibility properties of $\varphi^{\minus 1}(1,G)$ for general $G$. Whether or not $\varphi^{\minus 1}(H,G)$ can be zero for $H\in\mathcal{L}(G)$ in general is also unknown.\\

Note that, for $G\neq1$, one has $0=\delta(1,G)=\Sigma_{1\leq H\leq G}\,\varphi(1,H)\,\varphi^{\minus 1}(H,G)=\Sigma_{H\leq G,\,H\, \text{cyclic}}\,\varphi(|H|)\,\varphi^{\minus 1}(H,G)$ $=\Sigma_{g\in G}\,\varphi^{\minus 1}(\langle g\rangle,G)$. So, writing $\boldsymbol{\varphi_{\minus 1}}$ for the function $G\to \mathbb{Z},\,g\mapsto\varphi^{\minus 1}(\langle g\rangle,G)$, it follows that:
\begin{equation}\label{eq:phiinv0}
|G|>1\,\,\Rightarrow\,\,\Sigma_{g\in G}\,\varphi_{\minus 1}(g)\,=\,0
\end{equation}
\begin{proposition}\label{prop:phiinv}
The class function $\varphi_{\minus 1}$ is an integral linear combination of irreducible characters of G.
\end{proposition}
\begin{proof}
Let $\mathcal{K}=\{Hx\mid H\leq G,\,x\in G\}$ be the \textit{coset poset} of $G$, ordered by inclusion. (Usually, $G$ itself is not considered to be an element of the coset poset, but we include it here.) Every right coset $Hx=x(x^{\minus 1}Hx)$ is equally a left coset, and hence $\mathcal{K}$ becomes a $G$-set by putting $g\cdot (Hx):=gHx$ for $g\in G$. The action of $G$ is compatible with the ordering of $\mathcal{K}$. The fixed points of $g$ are the $Hx$ with $gHx=Hx$, that is, the right cosets of the subgroups $H$ of $G$ that contain the element $g$. We claim that:
\[
\varphi_{\minus 1}(g)=\Sigma_{H\leq G,\,g\in H}\,[G:H]\,\mu_{G}(H,G)=\Sigma_{\kappa\in\mathcal{K}^{g}}\,\mu_{g}(\kappa,G),
\]
where $\mathcal{K}^{g}$ denotes the set of fixed points of $g$ in $\mathcal{K}$ and $\mu_{g}$ is the M\"obius function of the poset $\mathcal{K}^{g}$. For if $H\backslash G$ is the set of right cosets of $H$ in $G$, one has $[G:H]\,\mu_{G}(H,G)=\Sigma_{\kappa\in H\backslash G}\,\mu_{\mathcal{K}}(\kappa,G)$, as noted by S. Bouc (cf. \cite[Section 3]{BCP}), and $g\in H$ iff all elements of $H\backslash G$ are in $\mathcal{K}^{g}$, and $\mu_{\mathcal{K}}(\kappa,G)=\mu_{g}(\kappa,G)$ for all $\kappa\in\mathcal{K}^{g}$ because $\mathcal{K}^{g}$ is a convex subset of $\mathcal{K}$.\\
Lemma 2.8 of \cite{HIO}, applied to the dual poset $\langle\mathcal{K},\supseteq\rangle$ of $\langle\mathcal{K},\subseteq\rangle$ (which has the "same" M\"obius function), now shows that the function $\varphi_{\minus 1}$ is a difference of permutation characters of $G$.
\end{proof}
By \eqref{eq:phiinv0}, the coefficient of the trivial character in the decomposition of $\varphi_{\minus 1}$ into irreducible characters is zero for non-trivial $G$.\\
\begin{proposition}\label{prop:invphi}
For $F\leq K\leq G$, the following formula holds:
\[
\varphi^{\minus 1}(K,G)\,=\,
\Sigma_{F\leq H\leq G,\,\langle H,K\rangle=G}\,\varphi^{\minus 1}(F,H)\,[G:H].
\]
\end{proposition}
\begin{proof}
First, we note that for $F\leq K\leq G$ one has
\begin{equation}\label{eq:invphi}
1=\zeta(F,K)=(\zeta*i^{\minus 1}*i)(F,K)=
(\varphi^{\minus 1}*i)(F,K)=
\Sigma_{F\leq H\leq K}\,\varphi^{\minus 1}(F,H)\,[K:H]
\end{equation}
We now prove the formula by induction on $[G:K]$. If $K=G$, its left-hand side is $\varphi^{\minus1}(G,G)=1$, and the right-hand side is the sum $\Sigma_{F\leq H\leq K}\,\varphi^{\minus 1}(F,H)\,[K:H]$, which in view of (\ref{eq:invphi}) also equals 1. Thus we may assume that $K<G$.\\
Using (\ref{eq:invphi}) plus the fact that for every sandwich $F\leq H\leq G$, the join $\langle H,K\rangle$ of $H$ and $K$ in the lattice $\mathcal{L}(G)$ is either $K$ itself or a subgroup $L>K$ of $G$, with either $L<G$ or $L=G$, we obtain
\begin{align*}
\,&\Sigma_{F\leq H\leq G,\,\langle H,K\rangle=G}\,\varphi^{\minus 1}(F,H)\,[G:H]=\\
\,&1-\Sigma_{F\leq H\leq K}\,\varphi^{\minus 1}(F,H)\,[G:H]
-\Sigma_{K<L<G}\,\Sigma_{F\leq H\leq L,\,\langle H,K\rangle=L}\,\varphi^{\minus 1}(F,H)\,[G:H]=\\
\,&1-[G:K]-\Sigma_{K<L<G}\,\varphi^{\minus 1}(K,L)\,[G:L].
\end{align*}
As to the second equality, the second terms on the last two lines agree by (\ref{eq:invphi}).
By the induction hypothesis, the third terms agree as well, seeing as, for $L\leq G$, the restrictions to $\mathcal{L}(L)\times\mathcal{L}(L)$ of the index and M\"obius functions on $\mathcal{L}(G)\times\mathcal{L}(G)$ are just the corresponding elements of $\mathcal{W}(L)$, so that the same goes for $\varphi=i*\mu$ and $\varphi^{\minus 1}$.\\
Now $\Sigma_{K<L<G}\,\varphi^{\minus 1}(K,L)\,[G:L]=(\varphi^{\minus 1}*i)(K,G)-\varphi^{\minus 1}(K,K)\,[G:K]-\varphi^{\minus 1}(K,G)\,[G:G]$, which is equal to $\zeta(K,G)-[G:K]-\varphi^{\minus 1}(K,G)=1-[G:K]-\varphi^{\minus 1}(K,G)$, and this establishes the result.
\end{proof}
The proposition provides an elegant liaison between the subgroups of $G$ that contain $K$ and the ones that together with $K$ generate $G$, when one applies it with $F=1$:
\[
\Sigma_{K\leq H\leq G}\,\zeta^{\minus 1}(H,G)\,[G:H]\,=\,
\Sigma_{H\leq G,\,\langle H,K\rangle=G}\,\varphi^{\minus 1}(1,H)\,[G:H].
\]

\end{document}